\documentclass[10pt]{article}
\usepackage{amsmath,amsfonts,amsthm,amssymb}

\theoremstyle{theorem}
\newtheorem{theorem}{Theorem}[section]
\newtheorem{proposition}[theorem]{Proposition}
\newtheorem{corollary}[theorem]{Corollary}
\newtheorem{lemma}[theorem]{Lemma}

\theoremstyle{definition}
\newtheorem{definition}[theorem]{Definition}

\newtheorem{example}[theorem]{Example}

\theoremstyle{remark}
\newtheorem{remark}[theorem]{Remark}

\newcommand{\al}{\alpha}
\newcommand{\be}{\beta}
\newcommand{\de}{\delta}

\newcommand{\ka}{\kappa}
\newcommand{\la}{\lambda}
\newcommand{\om}{\omega}

\newcommand{\vp}{\varphi}
\newcommand\vka{\kappa}

\newcommand\Om\Omega
\newcommand\Te\Theta
\newcommand{\De}{\Delta}
\newcommand{\Ga}{\Gamma}
\newcommand{\La}{\Lambda}
\newcommand{\Si}{\Sigma}
\newcommand{\hf}{\hat{f}}
\newcommand{\hg}{\hat{g}}
\newcommand{\htau}{\hat{\tau}}

\newcommand{\hF}{\widehat{F}}

\newcommand{\hcM}{\widehat{\mathcal M}}
\newcommand{\hcS}{\widehat{\mathcal S}}

\def\CC{\mathbb{C}}
\def\NN{\mathbb{N}}
\def\RR{\mathbb{R}}

\renewcommand\SS{\mathbb S}
\newcommand{\cD}{{\mathcal D}}

\newcommand{\cM}{{\mathcal M}}

\newcommand{\cR}{{\mathcal R}}
\newcommand{\cS}{{\mathcal S}}

\newcommand{\pd}{\partial}
\newcommand\minus\backslash

\newcommand{\ms}{\mspace{1mu}}
\newcommand\lan\langle
\newcommand\ran\rangle

\newcommand{\sign}{\operatorname{sign}}

\newcommand{\spec}{\operatorname{spec}}

\newcommand{\diam}{\operatorname{diam}}

\newcommand\Dom{\operatorname{Dom}}

\newcommand{\I}{{\mathrm i}}

\newcommand{\loc}{_{\mathrm{loc}}}
\newcommand{\dd}{{\mathrm d}}
\DeclareMathOperator\Real{Re}
\renewcommand\Re\Real
\renewcommand\leq\leqslant
\renewcommand\geq\geqslant
\renewcommand\Im{\operatorname{Im}}
\newcommand\sk{\sqrt k}
\newcommand\ska{\sqrt \vka}
\newcommand\dom{D}
\newcommand\hA{\hat A}
\newcommand\hB{\hat B}
\newcommand\ha{\hat a}
\newcommand\hb{\hat b}
\title{\bf Green's function for the Hodge Laplacian on some classes of Riemannian and Lorentzian symmetric spaces}
\author{Alberto Enciso$^1$\thanks{aenciso@fis.ucm.es} \and Niky Kamran$^2$\thanks{nkamran@math.mcgill.ca}}
\date{\small$^1$ Depto.~de F{\'\i}sica Te\'orica II, Universidad Complutense, 28040 Madrid, Spain\vspace{1ex}\\
$^2$ Dept.\ of Mathematics and Statistics, McGill University, Montr\'eal, Qu\'ebec, Canada H3A 2K6}

\begin{document}\maketitle

\begin{abstract}
We compute the Green's function for the Hodge Laplacian on the symmetric spaces $M\times\Si$, where $M$ is a simply connected $n$-dimensional Riemannian or Lorentzian
manifold of constant curvature and $\Si$ is a simply connected Riemannian surface of constant curvature. Our approach is based on a generalization to the case of
differential forms of the method of spherical means and on the use of Riesz distributions on manifolds. The radial part of the Green's function is governed by a fourth order analogue of the Heun equation.
\end{abstract}

\section{Introduction}

Our purpose in this paper is to compute the Green's function for the
Hodge Laplacian in some special classes of symmetric spaces of
Riemannian or Lorentzian signature, more precisely those spaces
which are products $M\times\Si$, where $M$ is a simply connected
$n$-dimensional Riemannian or Lorentzian manifold of constant
curvature and $\Si$ is a simply connected Riemannian surface of
constant curvature $\vka$~\cite{ON83}. The spaces we consider are
thus not of constant curvature and they are not conformally flat,
unless one of the factors reduces to a point. They contain in
particular the Robinson--Bertotti solutions of the Einstein--Maxwell
equations, as well as the de Sitter and anti-de Sitter spaces, the
latter spaces corresponding to cases in which $\Si$ reduces to a
point and $M$ is of Lorentzian signature.

The approach we take is based on a natural extension to the case of
differential forms of the classical method of spherical
means~\cite{Gu88}, which has been applied with great success to the
computation of the Green's function for the scalar wave equation in
symmetric spaces~\cite{He84}. The method of spherical means has the
advantage of leading naturally to the use of Riesz distributions,
which are fundamental analytical tools in the study of linear wave
equations on manifolds. Thus we avoid the need for special ans\"atze
on the form of the action of the Hodge Laplacian on the Green's
function, as is often the case in the Physics
literature~\cite{AJ86,Fo92}. In the case of spaces of constant
curvature, Riesz distributions provide a rigorous way of reducing
the computation of the Green's function to the study of an ordinary
differential equation governing its radial part. We will see that a
similar reduction occurs in the case of the product spaces
considered in this paper.

The computation of the Green's function for the Hodge Laplacian in
de Sitter and anti-de Sitter space is a problem of considerable
interest in Physics, where differential forms appear as basic
variables in various supersymmetric field theories, including
supergravity. In this context, the Green's function corresponds to
the propagator of the theory, which is the basic object of interest.
To our knowledge, the earliest results on this problem appear in
papers by Allen and Jacobson~\cite{AJ86} and Folacci~\cite{Fo91}.
Their results have been used in the recent works of a number of
physicists, e.g.~\cite{Ca92,Ha00,BU06}. The case of product
manifolds seems to have received less attention in the literature
and presents challenges of its own, as we shall see below. Beyond
the interest of this problem in Physics, our motivation has been to
illustrate the power of the method of spherical means as applied to
operators acting on sections of vector bundles, and to show that it
can be adapted efficiently to the case of product manifolds.

One of the main features of our result is that in the case in which
neither $\Si$ nor $M$ is reduced to a point, the expression for the
Green's function of the Hodge Laplacian acting on $p$-forms (with
$1\leq p \leq n-1$) has its radial part governed by a fourth order
ordinary differential equation with four regular singular points,
that is a ``higher order Heun equation'', which appears to be new.
This is in contrast with the case in which $\Si$ reduces to a point,
where the ordinary differential equations for the radial part of the
Green's function (which are coupled by the Levi-Civita connection
terms) can nevertheless be decoupled by a suitably chosen
differential substitution into second-order equations of
hypergeometric type. In all cases, we obtain a closed form
expression for the Green's function as a bundle-valued distribution.

There are some natural extensions that one could pursue for the
results that we have obtained in our paper. Of immediate interest
would be the extension of the method of computation of the Green's
function for the Hodge Laplacian to warped products of spaces of
constant curvature. One could also develop the formalism of our
paper in the case of the Laplacian on spinors, along the lines
of~\cite{AL86} either in the plain or warped product cases. A more
challenging objective, which would perhaps be of less interest in
Physics, would be to adapt to the case of differential forms of
arbitrary degree the analysis carried out by Lu~\cite{Lu91} for the
Green's function for the Laplacian on 1-forms over symmetric spaces
which are not maximally symmetric, such as the complex hyperbolic
spaces and some classical symmetric domains of higher rank.

Our paper is organized as follows. In Section \ref{S:sphericalRiem},
we briefly recall the definition and key properties of the spherical
means operators on differential forms defined on Riemannian space
forms, as well as their relation to the action of the Hodge
Laplacian. Section \ref{S:sphericalLor} contains a concise account
of the method of spherical means in the Lorentzian case, and
introduces the representation of the fundamental solution of the
wave equation in Lorentzian space forms in terms of Riesz
distributions. In Section \ref{S:Riem}, we consider the case of
Riemannian products, where we establish in Theorem \ref{T:mainR}
the expression of the Green's function for the Hodge Laplacian in
terms of the spherical means operators and the solutions of a
fourth-order ordinary differential equation with four regular
singular points, that is a ``fourth-order Heun equation'' which has
not been studied in the literature as far as we know. An important
element in our proof is a spectral decomposition result (Proposition
\ref{P:spectral}) for the second-order operator of singular
Sturm-Liouville type corresponding to the radial part of the Hodge
operator on surfaces, which we prove in Section \ref{S:dim2}. In
Section \ref{S:Lor}, we establish using Riesz distributions
the expression of the Green's function for the Hodge Laplacian in the
case of Lorentzian products. This is done in Theorem
\ref{T:mainLor}, and also involves a fourth-order radial equation
of Heun type which arises now from the use of Riesz potentials.
Finally, in Section \ref{S:constant}, we briefly show for the sake
of completeness that in the case in which the surface $\Si$ reduces
to a point, that is when the manifold is a simply connected
Riemannian or Lorentzian space form, the Green's function obtained
through our approach is expressible in terms of the solutions of a
second-order equation of hypergeometric type, as shown in
\cite{AJ86} and \cite{Fo92}.

\section{Spherical means on differential forms defined on Riemannian space forms}
\label{S:sphericalRiem}

In this section we shall briefly review some results on spherical
means for differential forms defined on the simply connected
Riemannian space form $M$ of constant sectional curvature $k$.
Proofs and further details can be found in Ref.~\cite{Gu88}. We
denote by $B_M(x,r)$ the ball centered at the point $x\in M$ of
radius $r$ and set $S_M(x,r)=\pd B_M(x,r)$. It is well known that
$S_M(x,r)$ is diffeomorphic to a sphere for any $0<r<\diam(M)$ and
that these spheres foliate $M\minus\{x\}$ (minus the antipodal point
of $M$ is positively curved). We use the notation $\dd S$ for the
induced area measure on $S_M(x,r)$, whose total area is
\begin{equation}\label{area}
m(r):=\big|S(x,r)\big|=\big|\SS^{n-1}\big|\Bigg(\frac{\sin \sk\, r}{\sk}\Bigg)^{n-1}\,.
\end{equation}
Throughout this paper we shall use the determination of the square
root with nonnegative imaginary part, which has a branch cut on the
positive real axis and is holomorphic in $\CC\minus[0,+\infty)$. We
also recall that the injectivity radius of $M$ coincides with
its diameter, which is $+\infty$ if $k\geq0$ and
$\pi/\sqrt{-k}$ if $k<0$. The space of smooth $p$-forms (resp.\ of
compact support) on $M$ will be denoted by $\Om^p(M)$ (resp.\
$\Om^p_0(M)$).

Let $\rho\in C^0(M\times M)$ be the distance function on $M$. For any degree $p$ we define double differential $p$-forms $\tau_p,\htau_p$ by setting
\begin{subequations}\label{tau}
\begin{align}
&\tau_0(x,x'):=1\,,\qquad\tau_1(x,x'):=-\frac{\sin\sk\,r}{\sk}\,\dd\dd'\rho(x,x')\,,\\
&\tau_p(x,x'):=\frac1p\tau_{p-1}(x,x')\wedge\!\wedge'\;\tau_1(x,x')\quad\text{for }\;2\leq p\leq n
\end{align}
and
\begin{align}
&\htau_0(x,x'):=0\,,\qquad\htau_1(x,x'):=-\dd \rho(x,x')\;\dd'\rho(x,x')\,,\\
&\htau_p(x,x'):=\htau_{1}(x,x')\wedge\!\wedge'\;\htau_{p-1}(x,x')\quad\text{for }\;2\leq p\leq n
\end{align}
\end{subequations}
at the points $(x,x')\in M\times M$ where the distance function is smooth. In the above formulas accented and unaccented operators act on $x'$ and $x$, respectively. Globally, this defines bundle-valued distributions $\tau_p,\htau_p\in\cD'(M\times M,\La^p(T^*M)\boxtimes \La^p(T^*M))$, where $\La^p(T^*M)\boxtimes \La^p(T^*M)$ is the vector bundle over $M\times M$ whose fiber at $(x,x')$ is $\La^p(T^*_xM)\otimes \La^p(T^*_{x'}M)$.

Given a smooth section $\Psi$ of $\La^p(T^*M)\boxtimes \La^p(T^*M)$ and $\om\in\Om^p(M)$, we introduce the notation
\begin{align*}
\Psi(x,x')\cdot\om(x)&:=*\big(\Psi(x,x')\wedge *\,\om(x)\big)\,,\\
\Psi(x,x')\cdot'\om(x')&:=*'\big(\Psi(x,x')\wedge' *'\om(x')\big)
\end{align*}
for the left and right inner products $\Gamma(\La^p(T^*M)\boxtimes \La^p(T^*M))\times\Om^p(M)\to C^\infty(M)\otimes\Om^p(M)$. Left and right products can be also defined for $\Psi\in\cD'(M\times M,\La^p(T^*M)\boxtimes \La^p(T^*M))$ and $\om\in\Om^p(M)$.

An important property of the double differential forms defined above
is that if $x'$ does not belong to the cut locus of $x$, then the
sum $\tau_p(x,x')+\htau_p(x,x')$ provides the parallel transport
operator for $p$-forms along the unique minimal geodesic connecting
$x$ and $x'$. More precisely, consider $\om_0\in\La^p(T^*_xM)$ and
let $\gamma$ be the minimal geodesic with $\gamma(0)=x$ and
$\gamma(1)=x'$. Then
$[\tau_p(x,x')+\htau_p(x,x')]\cdot\om_0=\tilde\om(1)$, where the
smooth map $\tilde\om:[0,1]\to\La^p(T^*M)$ satisfies the
differential equation for parallel transport along $\gamma$, which
with a slight abuse of notation can be written as
\[
\nabla_{\dot\gamma(t)}\tilde\om(t)=0\,,\qquad \tilde\om(0)=\om_0\,.
\]
In particular, $\tau_p(x,x)+\htau_p(x,x)$ is the identity map in $\La^p(T^*_xM)$.

\begin{definition}\label{D:RSM}
Let $0<r<\diam(M)$. The (Riemannian) {\em spherical means} of a
smooth $p$-form  $\om \in \Om^p(M)$ on the sphere of radius $r$ are
defined as
\begin{subequations}\label{cM}
\begin{align}
\cM_r\om(x)&:=\frac{1}{m(r)}\int_{S(x,r)}\tau_p(x,x')\cdot'\om(x')\,\dd
S(x')\,,\\
\hcM_r\om(x)&:=\frac{1}{m(r)}\int_{S(x,r)}\htau_p(x,x')\cdot'\om(x')\,\dd
S(x')\,.
\end{align}
\end{subequations}
\end{definition}

\begin{example}\label{ex1}
If $\om\in\Om^0(M)$, then $\hcM_r\om=0$ and $\cM_r\om$ coincides
with the usual spherical mean on constant curvature
spaces~\cite{He84}, which is customarily written as
\[
\cM_r\om(x)=\frac1{|\SS^{n-1}|}\int_{\SS^{n-1}} \om(x+r\theta)\,\dd S(\theta)
\]
in Euclidean space. If $\om\in\Om^n(M)$, $\cM_r\om=0$ and $*(\hcM_r\om)=\cM_r(*\ms\om)$.
\end{example}

We shall hereafter denote by $\De_M:=\dd\,\dd^*+\dd^*\dd$ the
(positive) Hodge Laplacian in $M$ and consider the measure on
$(0,\diam(M))$ given by $\dd\mu_M(r):=m(r)\,\dd r$. From the point
of view of their applications to the Hodge Laplacian, the key
property of the spherical means is given by the
relation~\cite[p.~47]{Gu88}
\begin{multline}\label{keyRiem}
\De_M\int_0^{\diam(M)}\Big(f(r)\,\cM_r\om+\hf(r)\,\hcM_r\om\Big)\,\dd\mu_M(r)=\\
\int_0^{\diam(M)}\Big(F(f,\hf)(r)\,\cM_r\om+\hF(f,\hf)(r)\,\hcM_r\om\Big)\,\dd\mu_M(r)\,,
\end{multline}
where $\om\in\Om^p(M)$ and $f,\hf$ are arbitrary smooth functions for which the latter integrals converge. The functions $F(f,\hf),\hF(f,\hf)$ are given by
\begin{subequations}\label{Fs}
\begin{align}
F(f,\hf)(r)&:=L_Mf(r)+pk\Big[\big(2\csc^2\sk r+n-p-1\big)f(r)-2\cos\sk
r\csc^2\sk r\,\hf(r)\Big]\,,\\
\hF(f,\hf)(r)&:=L_M \hf(r)+(n-p)k\Big[\big(2\csc^2\sk
r+p-1\big)\hf(r)-2\cos\sk r\csc^2\sk r\,f(r)\Big]\,,
\end{align}
\end{subequations}
$L_M$ being the radial Laplacian
\begin{equation}\label{LM}
-L_M:=\frac{\pd^2}{\pd r^2}+(n-1)\sk\cot\sk r\,\frac\pd{\pd r}\,.
\end{equation}
This formula remains valid for $f\in L^1_{\rm loc}((0,\diam(M)),\dd\mu_\Si)$ if $\om$ is compactly supported and the above derivatives and integrals are understood in the sense of distributions.

\section{The method of spherical means in the Lorentzian case and Riesz potentials}
\label{S:sphericalLor}

In this section we shall review the extension of the method of
spherical means to Lorentzian spaces of constant curvature and some
useful properties of Riesz potentials. Again we refer to~\cite{Gu88}
for details. We shall denote by $M$ the simply connected Lorentzian
manifold of constant sectional curvature $k$ and signature
$(-,+,\cdots,+)$. The Minkowski ($k=0$) and de Sitter spaces ($k>0$)
are globally hyperbolic, but the anti-de Sitter space ($k<0$) is not
(it is, however, strongly causal~\cite{ON83}). For this reason it is
convenient to consider a domain $\dom\subset M$ which will be fixed
during the ensuing discussion. When $k\geq0$, $\dom$ can be taken to
be the whole space $M$, whereas when $k<0$ we shall assume that
$\dom$ is a geodesically normal domain of $M$~\cite{Gu88}, that is,
a domain which is a normal neighborhood of each of its points.

We introduce the notation $J^+_\dom(x)$ for the points in $\dom$
causally connected with $x$, i.e., the points $x'\in\dom$ such that
there exists a future directed, non-spacelike curve from $x$ to
$x'$. Furthermore, we set
\begin{align*}
S^+_\dom(x,r):=\big\{x'\in J_\dom^+(x):\rho(x,x')=r\big\}
\end{align*}
and call $\dd S$ its induced area measure. Obviously
\begin{equation}\label{foliationLor}
J^+_\dom(x)=\bigcup_{r\geq0}S_\dom^+(x,r)\,,\qquad \text{with }\; S^+_\dom(x,r)\cap S^+_\dom(x,r')=\emptyset\;\text{ if }\;r\neq r'\,.
\end{equation}
We also use the notation
$\rho:\dom\times\dom\to[0,+\infty)$ for the Lorentzian distance
function.

\begin{definition}\label{D:Riesz}
The {\em Riesz potential} at $x$ with parameter $\al$ is the distribution $R^\al_{\dom,x}\in\cD'(\dom)$ defined for $\Re\al\geq n$ by the integral
\begin{equation}\label{Riesz}
R^\al_{\dom,x}(\phi):=C_\al\int_{J^+_\dom(x)}\rho(x,x')^{\al-n}\phi(x')\,\dd
x'\,,
\end{equation}
where $\phi$ is an arbitrary smooth function compactly supported in
$\dom$, $\dd x'$ stands for the Lorentzian volume element and we have
set
\begin{equation}\label{Cal}
C_\al:=\frac{\pi^{1-\frac
n2}2^{1-\al}}{\Gamma(\frac\al2)\Gamma(\frac{\al-n}2+1)}
\end{equation}
\end{definition}

For the sake of completeness, let us recall that the map $\al\mapsto
R^\al_{\dom,x}$ is holomorphic in $\CC$ provided that $\al\mapsto
R^\al_{\dom,x}(\phi)$ is an entire function for any smooth function
$\phi$ with compact support in $\dom$. A basic result in the theory
of Riesz potentials is the following theorem~\cite{Ri51,Gu88}.

\begin{theorem}\label{T:Riesz}
For any $x\in\dom$, the map $\al\mapsto R^\al_{\dom,x}$ can be holomorphically extended to the whole complex plane. Moreover, $\De_M R^\al_{\dom,x}=R^{\al-2}_{\dom,x}$ and $R^0_{\dom,x}=\de_x$.
\end{theorem}

\begin{corollary}\label{C:Riesz}
For any $\phi\in C^\infty_0(\dom)$, $u(x):=R_{\dom,x}^2(\phi)$ solves the hyperbolic equation $\De_Mu=\phi$ in $\dom$.
\end{corollary}

We define the bundle-valued distributions $\tau_p,\htau_p$ over
$\dom$ using the Lorentzian distance $\rho$ as in
Section~\ref{S:sphericalRiem}. If $x\in\dom$ and $x'\in J^+_\dom(x)$
does not lie on the causal cut locus of $x$,
$\tau_p(x,x')+\htau_p(x,x')$ is again the parallel transport
operator $\La^p(T^*_x\dom)\to \La^p(T^*_{x'}\dom)$ along the unique
future-directed minimal geodesic from $x$ to $x'$. In particular,
$\tau_p(x,x)+\htau_p(x,x)$ is the identity map on $\La^p(T^*_xM)$.

\begin{definition}\label{D:LSM}
Let $\om\in\Om^p_0(\dom)$. Its (Lorentzian) {\em spherical means} of radius $r$ are defined as
\begin{align*}
\cM_r^\dom\om(x)&:=\frac{(-1)^p}{m(r)}\int_{S^+_\dom(x,r)}\tau(x,x')\cdot'\om(x')\,\dd
S(x')\,,\\
\hcM_r^\dom\om(x)&:=\frac{(-1)^p}{m(r)}\int_{S^+_\dom(x,r)}\htau(x,x')\cdot'\om(x')\,\dd
S(x')\,,
\end{align*}
where $x\in\dom$ and $m$ is given by Eq.~\eqref{area}.
\end{definition}

We shall henceforth drop the scripts $\dom$ when there is no risk of
confusion. It should be noticed that the latter operators do not
exactly describe a spherical mean in the strict sense since in the
Lorentzian case $m(r)$ does not yield the area of $S^+(x,r)$, which
can be infinite. However, the fundamental relation
\begin{multline}\label{keyLor}
\De_M\int_0^\infty\Big(f(r)\,\cM_r\om+\hf(r)\,\hcM_r\om\Big)\,\dd\mu_M(r)=
\int_0^\infty\Big(F(f,\hf)(r)\,\cM_r\om+\hF(f,\hf)(r)\,\hcM_r\om\Big)\,\dd\mu_M(r)\,,
\end{multline}
also holds true in the Lorentzian case, with $F(f,\hf)$ and
$\hF(f,\hf)$ defined by Eq.~\eqref{Fs}.

\section{Riemannian products}
\label{S:Riem}

In this section we shall solve the Poisson equation
\begin{equation}\label{Poisson}
\De_{M\times\Si}\psi=\om\,,
\end{equation}
$\om$ being a compactly supported differential form in the Riemannian product space $M\times\Si$. If $M\times\Si$ is compact, we assume that $\om$ is orthogonal to the harmonic forms of the same degree in $M\times\Si$, which is the necessary and sufficient solvability condition of the latter equation~\cite{Wa78}.

It obviously suffices to analyze Eq.~\eqref{Poisson} when $\om\in\Om^p(M)\otimes\Om^q(\Si)$ for some integers $p$ and $q$. If we denote by $\Om_M$ and $\Om_\Si$ the Riemannian volume forms in $M$ and $\Si$, respectively, the above orthogonality conditions amount to imposing
\begin{subequations}\label{solv}
\begin{align}
&\int \om \,\Om_M\wedge\Om_\Si=0\quad\text{if }(p,q)=(0,0)\,,& \int \om \wedge\Om_M=0\quad\text{if }(p,q)=(0,2)\,,\\
&\int \om \wedge\Om_\Si=0\quad\text{if }(p,q)=(n,0)\,, &\int \om =0\quad\text{if }(p,q)=(n,2)\,\phantom{,}
\end{align}
\end{subequations}
when $M\times\Si$ is compact.

We begin by discussing these exceptional cases.  The explicit solution of the scalar equation ($p=q=0$) for functions $\om$ in $M\times\Si$ (and more general symmetric spaces) is well known and can be found in~\cite{He84}. Concerning the other exceptional cases, let us denote by $*_M$ and $*_\Si$ the Hodge star operator of $M$ and $\Si$ and observe that $*_\Si\om$, $*_M\om$ or $*_\Si(*_M\om)$ are $C^\infty$ scalar functions on $M\times\Si$, respectively, when $(p,q)=(0,2)$, $(n,0)$ or $(n,2)$. As $*_M$ and $*_\Si$ commute with the Hodge Laplacian of $M\times\Si$ and preserve the orthogonality relations~\eqref{solv}, for $(p,q)\in\{(0,2),(n,0),(n,2)\}$ Eq.~\eqref{Poisson} turns out to be equivalent to the scalar case by Hodge duality. Therefore we can henceforth assume that $(p,q)\not\in\{(0,0),(0,2),(n,0),(n,2)\}$. For concreteness we shall also assume that $k$ and $\vka$ are nonzero; the case where one of the curvatures is zero can be treated along the same lines (and is considerably easier).

We find it convenient to introduce the notation $\cS_s$ and $\hcS_s$ for the spherical means of
radius $s$ in the Riemannian surface $\Si$ and call
\[
\dd\mu_\Si(s):=2\pi\frac{\sin \ska\, s}{\ska}\,\dd s
\]
the radial measure on $(0,\diam(\Si))$. It should be observed that the decomposition
\[
\Om^{*}(M\times\Si)=\bigoplus_{p,q\geq0}\Om^p(M)\otimes\Om^q(\Si)
\]
defines a natural action of the Laplacians and spherical means
operators of $M$ and $\Si$ on $\Om^{*}(M\times\Si)$.

A simple but important observation is the following

\begin{lemma}\label{L:dim2}
Let $\be\in\Om^q_0(\Si)$. Then
\[
\De_\Si\int_0^{\diam(\Si)}w(s)\,(\cS_s+\hcS_s)\be\,\dd\mu_\Si(s)=\int_0^{\diam(\Si)}L_qw(s)\,(\cS_s+\hcS_s)\be\,\dd\mu_\Si(s)\,,
\]
where $w\in H^2\loc((0,\diam(\Si)),\dd\mu_\Si)$ and
\begin{subequations}\label{Lq}
\begin{align}
L_0=L_2&:=-\frac{\pd^2}{\pd s^2}-\ska\cot\ska s\,\frac\pd{\pd s}\,,\label{L0}\\
L_1w(s)&:=L_0w(s)+\frac{2\ska\,w(s)}{1+\cos\ska s}\,.\label{L1}
\end{align}
\end{subequations}
\end{lemma}
\begin{proof}
We saw in Example~\ref{ex1} that $\hcS_s\om=0$ when $q=0$. Hence in this case we easily obtain the desired formula from Eq.~\eqref{keyRiem} by noticing that $L_0w=F(w,w)$. When $q=1$, this follows directly from Eq.~\eqref{keyRiem} since $L_1w=F(w,w)=\hF(w,w)$. The case $q=2$ is essentially equivalent to $q=0$ by Hodge duality and uses that $L_2w=\hF(w,w)$.
\end{proof}

The formal differential operators $L_q$ given by~\eqref{Lq} define self-adjoint operators (which we still denote by $L_q$ for simplicity of notation) with domains given by the functions $u\in H^1((0,\diam(\Si)),\dd\mu_\Si)$  such that $L_qu\in L^2((0,\diam(\Si)),\dd\mu_\Si)$, $\lim\limits_{s\downarrow0}s\,u'(s)=0$, and
\begin{equation*}
\lim_{s\uparrow\diam(\Si)}\big(\diam(\Si)-s\big)\, u'(s)=0\;\; \text{if }\,\vka>0\,\text{ and }\,q=0\,.
\end{equation*}
Next we state a useful proposition on the spectral decomposition of the operators $L_q$. The proof, which consists of an application of the Weyl--Kodaira theorem and some manipulations of hypergeometric functions, is postponed until in Section~\ref{S:dim2}.

\begin{proposition}\label{P:spectral}
Let $c_{\vka,q}:=-\frac14\vka$ if $\vka<0$ and $c_{\vka,q}:=2\vka\de_{q1}$ if $\vka>0$. Then for each $q=0,1,2$ there exist a Borel measure $\rho_q$ on $[c_{\vka,q},\infty)$ and a $(\mu_\Si\times\rho_q)$-measurable function $w_q:(0,\diam(\Si))\times[c_{\vka,q},\infty)\to0$ such that:
\begin{enumerate}
\item $w_q(\cdot,\la)$ is an analytic formal eigenfunction of $L_q$ with eigenvalue $\la$  for $\rho_q$-almost every $\la\in [c_{\vka,q},\infty)$

\item The map
\[
u\mapsto \int_0^{\diam(\Si)} u(s)\,\overline{w_q(s,\cdot)}\,\dd\mu_\Si(s)
\]
defines a unitary transformation $U_q:L^2((0,\diam(\Si)),\dd\mu_\Si)\to L^2([c_{\vka,q},\infty),\dd\rho_q)$ with inverse given by
\[
U_q^{-1}u:=\int_{c_{\vka,q}}^{\infty} w_q(\cdot,\la)\,u(\la)\,\dd\rho_q(\la)\,.
\]

\item If $g(L_q)$ is any bounded function of $L_q$, then
\begin{equation}\label{L2conv}
g(L_q)u=\int_{c_{\vka,q}}^{\infty} g(\la)\,w_q(\cdot,\la)\,U_qu(\la)\,\dd\rho_q(\la)
\end{equation}
in the sense of norm convergence.
\end{enumerate}
\end{proposition}
\begin{remark}
Explicit formulas for all the functions involved are given in Lemmas~\ref{L:spec+} and~\ref{L:spec-}. By definition, $\rho_q$ is supported on the spectrum of the self-adjoint operator $L_q$.
\end{remark}

We shall construct a solution to Poisson's equation~\eqref{Poisson} as
\begin{align}\label{ansatzRiem}
\psi=\int w_q(s,\la)\,(\cS_s+\hcS_s)\big(f(r,\la)\cM_r+\hf(r,\la)\hcM_r\big)\om\,\dd\mu_\Si(s)\,\dd\mu_M(r)\,\dd\rho_q(\la)\,,
\end{align}
where the integral ranges over $[c_{\vka,q},\infty)\times(0,\diam(M))\times(0,\diam(\Si))$.
By construction, $L_q w_q(s,\la)=\la\, w_q(s,\la)$ for $\rho_q$-almost every
$\la$, so from Lemma~\ref{L:dim2} and the fact that
$\De_{M\times\Si}\psi=\De_M\psi+\De_\Si\psi$ it follows that
\begin{align}\label{LaplacianMSi}
\De_{M\times\Si}\psi=\int w_q\,(\cS+\hcS)\big[\big(F(f,\hf)+\la f\big)\cM+\big(\hF(f,\hf)+\la\hf\big)\hcM\big]\om\,\dd\mu_\Si\,\dd\mu_M\,\dd\rho_q\,,
\end{align}
with $F$ and $\hF$ given by~\eqref{Fs}.

Let us consider the system of coupled ODEs
\begin{equation}\label{ODEs}
F(f,\hf)+\la f=0\,,\qquad \hF(f,\hf)+\la\hf=0\,,
\end{equation}
depending on the nonnegative parameter $\la$. Let us introduce the variable $z:=\sin^2\frac{\sk r}2$, which takes values in $(0,1)$ when $r\in(0,\diam(M))$ and $M$ is compact and in $(-\infty,0)$  when $M$ is noncompact. With an abuse of notation we shall temporarily write $f(z)$ or $\hf(z)$ for the expression of the functions $f$ or $\hf$ in this variable, that is, for $f(\frac2{\sk}\arcsin\sqrt{\cdot})$. This allows to write Eqs.~\eqref{ODEs} as
\begin{subequations}\label{ODEsz}
\begin{align}
z(1-z)f''(z)+\frac n2(1-2z)f'(z)&=\bigg[\ell+p\bigg(\frac1{2z(1-z)}+n-p-1\bigg)\bigg]f(z)-\frac{p(1-2z)}{2z(1-z)}\hf(z)\,,\label{ODEsza}\\
z(1-z)\hf''(z)+\frac n2(1-2z)\hf'(z)&=\bigg[\ell+(n-p)\bigg(\frac1{2z(1-z)}+p-1\bigg)\bigg]\hf(z)-\frac{(n-p)(1-2z)}{2z(1-z)}f(z)\label{ODEszb}\,,
\end{align}
\end{subequations}
with $\ell:=\la/k$. We can solve~\eqref{ODEsza} for $\hf$ and substitute in the second equation to arrive at a linear fourth order differential equation for $f$, namely
\begin{equation}\label{fourth}
\frac{\dd^4f}{\dd z^4}+\sum_{j=0}^3Q_j(z)\,\frac{\dd^jf}{\dd z^j}=0
\end{equation}
with the rational functions $Q_j$ being given by
\begin{align*}
Q_3(z)&=\frac{n+4}{z-1}+\frac{n+4}{z}+\frac{4}{1-2 z}\,,\\
Q_2(z)&=\Big[16\big(n^2+2 (p+3) n-2 p^2+2 (\ell +3)\big) z^4+32\big(2 p^2-n^2-2 (p+3) n\\
&-2 (\ell +3)\big) z^3
+8\big(3 n^2+5 (p+3) n+3 n-5 p^2+
   \ell +4 (\ell +3)+6\big) z^2+8\big(p^2-n^2\,,\\
&- (p+3) n-3 n- \ell -6\big) z+n^2+6 n+8\Big]\Big/\Big(16 z^6-48 z^5+52 z^4-24 z^3+4 z^2\Big)\\
Q_1(z)&=\Big[2 p n^2+n^2-2 p^2 n+4 p n+2 \ell  n+2 n-4 p^2+4 \ell +4 \big(-2 p n^2- n^2+2 p^2 n\\
&-2 p n-2 \ell  n-2 n+2 p^2+2 p-2 \ell \big)z+4 \big(2 p
   n^2+ n^2-2 p^2 n+2 p n+2 \ell  n+2 n-2 p^2\\
&-2 p+2 \ell \big)z^2\Big]\Big/\Big(4 z^5-10 z^4+8 z^3-2 z^2\Big)\,,\\
Q_0(z)&=\Big[p^4-2 n p^3+n^2 p^2-2 \ell  p^2+p^2-2 n p+2 n \ell  p-2 p+\ell ^2-2 \ell + 4\big(- p^4+2 n p^3- n^2 p^2\\
&+2 \ell  p^2+ p^2-2 n \ell  p-
   \ell ^2\big)z+ 4\big( p^4-2 n p^3+ n^2 p^2-2 \ell  p^2- p^2+2 n \ell  p\\
&+ \ell ^2\big)z^2 \Big]\Big/\Big(4 z^6-12 z^5+13 z^4-6 z^3+z^2\Big)\,.
\end{align*}

The fourth order equation~\eqref{fourth} plays a crucial role in the computation of the Green's function of the Hodge Laplacian. A simple computation shows that it has four regular singular points at $0$, $\frac12$, $1$ and $\infty$, so it can be understood as a fourth order analogue of the Heun equation.
For our purpose it is important to make the following observation.

\begin{proposition}\label{P:fourth1}
For all $\ell\leq0$ there exists a unique real solution $\Phi_+(z,\ell)$ of Eq.~\eqref{fourth} with the asymptotic behavior at $0$
\begin{equation}\label{at0}
\Phi_+(z,\ell)=\frac{4z^{1-\frac n2}}{k(n-2)|\SS^{n-1}|}+O(z^{2-\frac n2})\quad \text{if }\,n\geq3\quad\text{and}\quad \Phi_+(z,\ell)=-\frac{\log |z|}{4\pi}+O(1)\quad \text{if }\,n=2\,,
\end{equation}
and the fastest possible decay as $z$ tends to $-\infty$. If $\ell>0$ or $p\neq0,n$ and $\ell\geq0$, there is also a unique real solution $\Phi_-(z,\ell)$ of the latter equation such that  $z^{\frac n2-1}\Phi_-(z,\ell)$ is analytic in the closed interval $[0,1]$ and has the asymptotic behavior~\eqref{at0} at $0$.
\end{proposition}
\begin{proof}
Eq.~\eqref{fourth} is a fourth order Fuchsian equation with poles at $0$, $\frac12$, $1$ and $\infty$. Suppose that $n\geq3$. Then characteristic exponents at $0$ and $1$ are $-\frac n2,1-\frac n2,0$ and $1$, so that $\Phi_+$ must be given by
\begin{equation*}
\Phi_+(z,\ell)=\frac{4}{k(n-2)|\SS^{n-1}|}\phi_{1-\frac n2}(z)+c_1\phi_{0}(z)+c_2\phi_{1}(z)\,,
\end{equation*}
where $\phi_{\nu}$ stands for the local solution of~\eqref{fourth} which is analytic on $(-\infty,0)$ and asymptotic to $z^\nu$ at $0$. The characteristic exponents of at $\infty$ are $\frac12[n+1\pm((n+1-2p)^2-4\ell)^{1/2}] $ and $\frac12[n-1\pm((n-1-2p)^2-4\ell)^{1/2}]$, so that each local solution $\phi_\nu$ has well-defined asymptotics at $\infty$ and the real constants $c_1$ and $c_2$ can be chosen so as to obtain the fastest possible decay. When $n=2$, the local solutions at zero behave as $z^{-1}$, $\log|z|$, $1$ and $z$ and the same reasoning applies mutatis mutandis.

The analysis on the interval $(0,1)$ is similar. The characteristic exponents at $\frac12$ are $0,1,3$ and $4$, so all the solutions of the equation are analytic at this point. The function $\Phi_-$ must therefore be given by
\begin{equation*}
\Phi_-(z,\ell)=\frac{4}{k(n-2)|\SS^{n-1}|}\phi_{1-\frac n2}(z)+c_1\phi_{0}(z)+c_2\phi_{1}(z)\,,
\end{equation*}
where the local solutions $\phi_\nu$ are now analytic in $(0,1)$ and asymptotic to $z^\nu$ at $0$. The constants $c_1$ and $c_2$ should now be chosen so as to ensure that $\Phi_-$ does not have any terms asymptotic to $z^{-\frac n2}$ or $z^{1-\frac n2}$. With some effort it can be shown that this is always possible when $p\neq0,n$ or $\ell\neq0$, which is not surprising since these conditions are directly related to the solvability of the associated equation $(k\ell+\De_M)\Psi=\al$, with $M$ compact, $\al\in\Om^p(M)$ and $k\ell\geq0$.
\end{proof}

Now we have all the ingredients to prove the main result of this section.

\begin{theorem}\label{T:mainR}
Let us set $f(r,\la):=\overline{w_q(0,\la)}\,\Phi_{\sign k}(\sin^2\frac{\sk r}2,\la/k)$ and define a function $\hf(r,\la)$ by means of Eq.~\eqref{ODEsza}, i.e., as
\[
\hf(r,\la):=\frac{L_Mf(r,\la)+pk\big(2\csc^2\sk r+n-p-1\big)f(r,\la)}{2pk\,\cos\sk
r\csc^2\sk r}\,.
\]
Then the function $\psi$ given by~\eqref{ansatzRiem} solves Eq.~\eqref{Poisson} for $\om\in\Om^p(M)\otimes\Om^q(\Si)$.
\end{theorem}
\begin{proof}
We start by noticing that, by Proposition~\ref{P:fourth1} and the restriction on the possible values of $(p,q)$, the function $\Phi_{\sign k}(\sin^2\frac{\sk r}2,\la/k)$ is well defined for all values of $r$ and $\la$ in the integration range. It should be noticed that the definition of $\hf(\cdot,\la)$ ensures that it has the same asymptotic behavior at $0$ that $f(\cdot,\la)$, namely~\eqref{at0}
\[
\frac{\overline{w_q(0,\la)}}{(n-2)|\SS^{n-1}|\,r^{n-2}}+O(r^{3-n})\,.
\]
An important observation is that the distribution defined by $F(f(\cdot,\la),\hf(\cdot,\la))\,\dd\mu_\Si$ (which equals $\hF(f(\cdot,\la),\hf(\cdot,\la))\,\dd\mu_\Si$) is in fact $\overline{w_q(0,\la)}$ times the Dirac delta supported at $0$, since
\begin{multline}\label{intM}
\int_0^{\diam(M)}\Big[f(r,\la)\, L_M\vp(r)+pk\Big(\big(2\csc^2\sk r+n-p-1\big)f(r,\la)\\-2\cos\sk
r\csc^2\sk r\,\hf(r,\la)\Big)\vp(r)\Big]\dd\mu_\Si(r)=\overline{w_q(0,\la)}\,\vp(0)
\end{multline}
for all $\vp\in C^\infty_0([0,\diam(M)])$. This immediately stems from the fact that $F(f(\cdot,\la),\hf(\cdot,\la))(r)$ is zero for all $r\neq0$ by construction (cf.\ Proposition~\ref{P:fourth1}) and the asymptotic behavior of $f$ and $\hf$ at zero. In particular, as $\cM_0+\hcM_0$ is the identity map it easily follows that
\begin{equation*}
\int_0^{\diam(M)} \big[\big(F(f,\hf)+\la f\big)\cM_r+\big(\hF(f,\hf)+\la\hf\big)\hcM_r\big]\om\,\dd\mu_M(r)=\om\,.
\end{equation*}
As a consequence of this, Eq.~\eqref{LaplacianMSi} reduces to
\begin{align}\label{intla}
\De\psi&=\int \overline{w(0,\la)}\, w(s,\la)\,\big(\cS_s+\hcS_s\big)\om\,\dd\mu_\Si(s)\,\dd\rho_q(\la)\,.
\end{align}
When the integral in $\la$ ranges over $[c_{q,\ka},\infty)$, it immediately follows from Proposition~\ref{P:spectral} and Lemma~\ref{L:unif} on the pointwise convergence of the integral~\eqref{L2conv} that
\begin{align*}
\De\psi&=\big(\cS_0+\hcS_0\big)\om=\om\,,
\end{align*}
and the claim in the statement follows.
\end{proof}
\begin{remark}
Being constructed using only the distance function and the metric on $M$ and $\Si$, the symmetric Green's function that we have constructed is equivariant under the isometries of $M\times\Si$.
\end{remark}
\begin{remark}
Some comments on the uniqueness of the solution to Eq.~\eqref{Poisson} are in order. When $M\times\Si$ is compact, it follows from standard Hodge theory that for the above values of $(p,q)$ we have constructed the only $L^2$ solution to the equation. The case when $M$ or $\Si$ is noncompact can be analyzed using that $\ker(\De_{M\times\Si})=\ker(\De_M)\otimes\ker(\De_\Si)$ and that the kernel of the Hodge Laplacian of the hyperbolic $m$-space acting on $s$-forms is $\{0\}$ if $n\neq 2s$ and infinite dimensional otherwise~\cite{Do79,Do80}. Hence for the above values of $(p,q)$ Theorem~\ref{T:mainR} also yields the only $L^2$ solution to~\eqref{Poisson} when $k>0$, $\vka<0$ and $(p,q)\not\in\{(0,1),(n,1)\}$, when $k<0$, $\vka>0$ and $(p,q)\not\in\{(\frac n2,0),(\frac n2,2)\}$, and when $k<0$, $\vka<0$ and $(p,q)\neq(\frac n2,1)$.
\end{remark}

\section{Lorentzian products}
\label{S:Lor}

In this section we shall solve the equation
\begin{equation}\label{wave}
\De_{M\times\Si}\psi=\om
\end{equation}
for a compactly supported form $\om\in\Om^p(\dom)\otimes\Om^q(\Si)$
by constructing an advanced Green's operator for $\De_{M\times\Si}$.
We assume that $M$ and $\Si$ respectively have Lorentzian and
Riemannian signature. As in Section~\ref{S:sphericalLor}, if $M$ is
not globally hyperbolic ($k<0$) we restrict ourselves to a
geodesically normal domain $\dom\times\Si$. If $M$ is globally
hyperbolic ($k\geq0$) we simply set $\dom:=M$. The symbols $\cM_r$
and $\hcM_r$ will stand for the Lorentzian spherical means in
$\dom$, but other than that we will use the same notation as in the
previous section. We recall that for any globally hyperbolic
manifold $\dom\times\Si$ there exists a unique advanced Green's
operator $\Om^*_0(\dom\times\Si)\to\Om^*(\dom\times\Si)$ for the
Hodge Laplacian~\cite{BGP07}.

As discussed in the previous section, when
$(p,q)\in\{(0,0),(0,2),(n,0),(n,2)\}$ Eq.~\eqref{wave} is equivalent
by Hodge duality to the scalar wave equation, whose Green's function
is well known~\cite{He84}. Hence we shall assume that $(p,q)$ does
not take any of the above values and that both $k$ and $\vka$ are
nonzero.

The results in the previous section and the definition of the Riesz potentials strongly suggest that we analyze the system
of ordinary differential equations
\begin{subequations}\label{FhFLor}
\begin{align}
F\big(f_\al(\cdot,\la),\hf_\al(\cdot,\la)\big)(r)+\la\,f_\al(r,\la)&=C_{\al-2}\, r^{\al-n-2}\,,\\ \hF\big(f_\al(\cdot,\la),\hf_\al(\cdot,\la)\big)(r)+\la\,\hf_\al(r,\la)&=C_{\al-2}\, r^{\al-n-2}\,,
\end{align}
\end{subequations}
where $\la$ is positive, $C_\al$ is given by~\eqref{Cal} and we
assume for the moment that $\Re\al> n+2$. We find it convenient to
write this equation in the variable $z:=\sin^2\frac{\sk r}2$,
writing $f_\al(z)$ or $\hf_\al(z)$ for the expression of the
functions $f_\al(r,\la)$ or $\hf_\al(r,\la)$ in terms of $z$ with
some abuse of notation. Defining the function
$h_\al(z):=-\frac1kC_{\al-2}\big(\frac2\sk\arcsin\sqrt
z\big)^{\al-n-2}$, Eqs.~\eqref{FhFLor} now read
\begin{subequations}\label{ODEszLor}
\begin{align}
z(1-z)f_\al''(z)+\frac n2(1-2z)f_\al'(z)&-\bigg[\ell+p\bigg(\frac1{2z(1-z)}+n-p-1\bigg)\bigg]f_\al(z)=\notag\\
&\hspace*{12em}=h_\al(z)-\frac{p(1-2z)}{2z(1-z)}\hf_\al(z)\,,\label{ODEszaLor}\\
z(1-z)\hf_\al''(z)+\frac n2(1-2z)\hf_\al'(z)&-\bigg[\ell+(n-p)\bigg(\frac1{2z(1-z)}+p-1\bigg)\bigg]\hf_\al(z)=\notag\\
&\hspace*{12em}=h_\al(z)-\frac{(n-p)(1-2z)}{2z(1-z)}f_\al(z)\label{ODEszbLor}\,,
\end{align}
\end{subequations}
where $\ell:=\la/k$. We can combine Eqs.~\eqref{ODEszLor} to obtain a single fourth order equation for $f_\al$, namely
\begin{equation}\label{fourthLor}
\frac{\dd^4f_\al}{\dd z^4}+\sum_{j=0}^3Q_j(z)\,\frac{\dd^jf_\al}{\dd z^j}=H_\al(z)\,,
\end{equation}
where the rational functions $Q_j$ were defined in Section~\ref{S:Riem},
\[
H_\al(z):=q_0(z)h_\al(z)+q_1(z)h_\al'(z)+q_2(z)h_\al''(z)
\]
and
\begin{align*}
q_2(z)&=\frac{2 (z-1)^2 z^2}{p (2 z-1)}\,,\qquad
q_1(z)= \frac{(z-1) z \left(4 n z^2+8 z^2-4 n z-8 z+n+4\right)}{p (2 z-1)^2}\,,\\
q_0(z)&=\frac{2 z \left(\ell (z-1) (1-2 z)^2+p (-z p+p+n (z-1)+z) (1-2 z)^2-2 z+2\right)}{p (2 z-1)^3}\,.
\end{align*}
Using the results from the previous section it is not difficult to prove the following

\begin{proposition}\label{P:PhiLor}
Under the above hypotheses, Eq.~\eqref{fourthLor} has a unique solution $\Phi_{\al,+}(z,\ell)$ which is continuous in $[0,1]$ if $k<0$, and a unique solution $\Phi_{\al,-}(z,\ell)$ which is continuous in $(-\infty,0]$ and has the fastest possible decay at infinity if $k<0$. These solutions are real for real $\al$.
\end{proposition}
\begin{proof}
Let us assume that $n\geq3$ and $k<0$. We omit the discussion of the case $n=2$, which goes along the same lines by replacing $z^{1-\frac n2}$ by $\log |z|$ in the discussion. We saw in Proposition~\ref{P:fourth1} that the homogeneous equation~\eqref{fourth} has four regular singularities at $0$, $\frac12$, $1$ and $\infty$. As $H_\al$ is continuous on $(-\infty,0]$ for $\al> n+2$, the method of the variation of constants yields a particular solution $\Phi_0$ of~\eqref{fourthLor} which is analytic in $(-\infty,0)$ with possible singularities at $0$ of order $z^{\frac n2}$ and $z^{1-\frac n2}$. The function in the statement of the theorem is thus given by
\begin{equation}\label{Prop4th}
\Phi_{\al,-}(z,\ell)=\Phi_0(z)+c_1\phi_{-\frac n2}(z)+c_2\phi_{1-\frac n2}(z)+c_3\phi_0(z)+c_4\phi_1(z)
\end{equation}
in the notation of Proposition~\ref{P:fourth1}. The constants $c_1$ and $c_2$ are chosen so that
\[
\lim_{z\uparrow 0}z^{\frac n2-1}\Phi_{\al,-}(z,\ell)=0\,,
\]
while $c_3$ and $c_4$ are chosen so as to obtain the fastest possible decay at infinity.

Suppose now that $k>0$. Let us observe that for $\al> n+2$ the function $H_\al$ is continuous in $[0,\frac12)\cup(\frac12,1]$, whereas it diverges as $(z-\frac12)^{-3}$ at $\frac12$. However, as the characteristic exponents of the homogeneous equation at $\frac12$ are $(0,1,3,4)$, it is standard that the method of the variation of constants yields a particular solution $\Phi_0$ of Eq.~\eqref{fourthLor} which is continuous (actually, analytic) at $\frac12$. The desired solution is obtained, in the notation of Proposition~\ref{P:fourth1}, as
\[
\Phi_{\al,+}(z,\ell)=\Phi_0(z)+c_1\phi_{-\frac n2}(z)+c_2\phi_{1-\frac n2}(z)+c_3\phi_0(z)+c_4\phi_1(z)\,.
\]
Here the constants $c_j$ are chosen so that
\[
\lim_{z\downarrow 0}z^{\frac n2-1}\Phi_{\al,+}(z,\ell)=\lim_{z\downarrow 0}(1-z)^{\frac n2-1}\Phi_{\al,+}(z,\ell)=0\,,
\]
i.e., so as to remove the singularities of order $\frac n2$ and $\frac n2-1$ at $0$ and at $1$. From the proof of Proposition~\ref{P:fourth1} it stems that this determines the constants $c_j$.
\end{proof}

The $\La^p(T_x^*M)$-valued distribution defined for each $x\in \dom$ as
\begin{align*}
\widetilde \cR^{\al}_{x,\dom}(\phi)&:=C_{\al}\int r^{\al-n}\big(\cM_r+\hcM_r\big)\phi(x)\,\dd\mu_M(r)\\
&=C_{\al}\int_{J^+_\dom(x)}\rho(x,x')^{\al-n}\big(\tau(x,x')+\htau(x,x')\big)\cdot' \phi(x')\,\dd x'\,,\qquad \phi\in\Om^p_0(M)\,,
\end{align*}
is easily seen to be holomorphic for $\Re\al>n$, with fixed $x$ and $\la$. By Theorem~\ref{T:Riesz} the function $\al\mapsto\widetilde \cR^{\al}_{x,\dom}$ admits an entire extension, and as $\tau(x,x)+\htau(x,x')={\rm id}_{\La^p(T^*_xM)}$ it follows from this theorem that
\begin{equation}\label{Rtilde}
\widetilde \cR^0_{x,\dom}(\phi)=\phi(x)\,.
\end{equation}

In what follows we shall set $f_\al(r,\la)=\Phi_{\al,\sign k}(\sin^2\frac{\sk r}2,\la/k)$, with $\Phi_{\al,\pm}$ as in Proposition~\ref{P:PhiLor}, and define $\hat f_\al$ so that Eq.~\eqref{FhFLor} holds true, that is, as
\[
\hf_\al(r,\la):=\frac{L_Mf_\al(r,\la)+pk\big(2\csc^2\sk r+n-p-1\big)f_\al(r,\la)-C_{\al-2}r^{\al-n-2}}{2pk\,\cos\sk
r\csc^2\sk r}\,.
\]
It is natural to define another vector-valued distribution by
\[
\cR^{\al,\la}_{x,\dom}(\phi):=\int\big(f_\al(r,\la)\,\cM_r+\hf_\al(r,\la)\hcM_r\big)\phi(x)\,\dd\mu_M(r)\,,\qquad\phi\in\Om^p_0(M)\,,
\]
for $\Re\al>n-2$. By Eqs.~\eqref{keyLor}, \eqref{FhFLor} and~\eqref{Rtilde} it follows that the two distributions that we have just defined are related by
\[
(\De_M+\la)\cR^{\al,\la}_{x,\dom}=\widetilde \cR^{\al-2}_{x,\dom}\,.
\]
In fact, from the above equation and Proposition~\ref{P:PhiLor} it is standard~\cite{Ri51,BGP07} that the function $\al\mapsto\cR^{\al,\la}_{x,\dom}$ can be holomorphically extended to the whole complex plane. Thus we are led to a useful generalization of Corollary~\ref{C:Riesz} that the summarize in the following

\begin{proposition}\label{P:RieszForms}
The distribution $\cR^{\al,\la}_{x,\dom}$ is a holomorphic function of $\al\in\CC$ for fixed $x\in\dom$, $\la\in\spec(L_q)$. Moreover, for any $\phi\in\Om^p_0(\dom)$ the differential form $\Phi(x):=\cR^{\al,\la}_{x,\dom}(\phi)$ solves the equation $(\De_M+\la)\Phi=\phi$ in $\dom$.
\end{proposition}

The action of $\cR^{\al,\la}_{x,\dom}$ naturally defines a bundle-valued distribution on $M\times\Si$, which we do not distinguish notationally, which acts on the first factor of any $\om\in\Om^p(\dom)\otimes\Om^q(\Si)$ to yield an element of $\La^p(T^*_xM)\otimes\Om^q(\Si)$. Namely, if $\om(x,y)=\om_1(x)\otimes\,\om_2(y)$ this action reads as
\[
\cR^{\al,\la}_{x,\dom}(\om)(y):=\cR^{\al,\la}_{x,\dom}(\om_1)\otimes\,\om_2(y)\,,
\]
with $(x,y)\in \dom\times\Si$. This result allows us to express the solution of Eq.~\eqref{wave} as follows.

\begin{theorem}\label{T:mainLor}
The differential form
\[
\psi(x,y):=\int \overline{w_q(0,\la)}\, w_q(s,\la)\,\cR^{\al,\la}_{x,\dom}(\om)(y)\,\dd\mu_M(r)\,\dd\mu_\Si(s)\,\dd\rho_q(\la)\,,
\]
where the integral ranges over $[c_{q,\vka},\infty)\times(0,\diam(\Si))\times(0,\diam(M))$, solves Eq.~\eqref{wave} for any compactly supported $\om\in\Om^p(\dom)\otimes\Om^q(\Si)$.
\end{theorem}
\begin{proof}
If immediately follows from the definition of $w_q$ and Proposition~\ref{P:RieszForms} that
\begin{align*}
\De_{M\times\Si}\psi(x,y)&=\int \overline{w_q(0,\la)}\, w_q(s,\la)\,(\la+\De_M)\cR^{\al,\la}_{x,\dom}(\om)(y)\,\dd\mu_M(r)\,\dd\mu_\Si(s)\,\dd\rho_q(\la)\\
&=\int \overline{w_q(0,\la)}\, w_q(s,\la)\,\om(x,y)\,\dd\mu_\Si(s)\,\dd\rho_q(\la)\,.
\end{align*}
As in the proof of Theorem~\ref{T:mainR}, now Proposition~\ref{P:spectral} and Lemma~\ref{L:unif} show that
\[
\De_{M\times\Si}\psi(x,y)=(\cS_0+\hcS_0)\om(x,y)=\om(x,y)\,,
\]
completing the proof of the statement.
\end{proof}

\section{Constant curvature spaces}
\label{S:constant}

We saw in Eqs.~\eqref{ODEs} and~\eqref{FhFLor} that the radial
behavior of the Green's operator in the product manifolds
$M\times\Si$ is controlled by a fourth order analogue of the Heun
equation. On the other hand, it is well known~\cite{AJ86,Fo92} that
the Hodge Green's function on $M$ (as well as the scalar Green's
function of any simply connected rank 1 symmetric space~\cite{He84})
is controlled by simple hypergeometric equations: this merely
reflects that the geometry of $M\times\Si$ is considerably more
involved than that of its individual factors. For the sake of
completeness, in this short section we shall discuss the
simplifications that give rise to hypergeometric functions when one
considers the equation on $M$, that is, when the factor $\Si$
collapses to a point.

We shall therefore analyze the following system of coupled ODEs:
\begin{subequations}\label{ODEszCCS}
\begin{align}
z(1-z)f''(z)+\frac n2(1-2z)f'(z)&-\bigg[\ell+p\bigg(\frac1{2z(1-z)}+n-p-1\bigg)\bigg]f(z)=\notag\\
&\hspace*{12em}=h(z)-\frac{p(1-2z)}{2z(1-z)}\hf(z)\,,\label{ODEszaCCS}\\
z(1-z)\hf''(z)+\frac n2(1-2z)\hf'(z)&-\bigg[\ell+(n-p)\bigg(\frac1{2z(1-z)}+p-1\bigg)\bigg]\hf(z)=\notag\\
&\hspace*{12em}=h(z)-\frac{(n-p)(1-2z)}{2z(1-z)}f(z)\label{ODEszbCCS}\,.
\end{align}
\end{subequations}
When $h(z)=0$ this equation controls the radial part of the Green's operator for the Poisson equation on $M\times\Si$ studied in Section~\ref{S:Riem}, whereas when $h$ is the function $h_\al$ defined in Section~\ref{S:Lor} we obtain the equation for the radial part of the retarded Green's function.
It is natural to look for functions $a,\ha,b$ and $\hb$ (possibly depending on $p$ and $\ell$) such that the functions
\begin{equation}\label{fg}
g(z):=f'(z)+p\ms\ms a(z)\ms f(z)-p\ms b(z)\ms\hf(z)\,,\qquad \hg(z):=\hf'(z)+(n-p)\ms\ha(z)\ms\hf(z)-(n-p)\ms\hb(z)\ms f(z)
\end{equation}
satisfy a first-order system of ODEs (equivalent to~\eqref{ODEszCCS}) of the form
\begin{subequations}\label{1st}
\begin{align}
z(1-z)\ms g'(z)+(n-p)A(z)\ms g(z)+pB(z)\ms \hg(z)&=h(z)\,,\\
z(1-z)\ms \hg'(z)+p\hA(z)\ms \hg(z)+(n-p)\hB(z)\ms g(z)&=h(z)\,.
\end{align}
\end{subequations}
When this reduction can be performed, it is possible to express the functions $f$ and $\hf$ in terms of the solutions to the decoupled second-order equations for $f,\hf$ and $g,\hg$ that one obtains from~\eqref{fg} and~\eqref{1st}. As before, by Hodge duality and the fact that the scalar case is well known allows us to assume that $p\neq0,n$.

After some manipulations one finds that these equations amount to imposing the following conditions:
\begin{align*}
&(z-1) z b(z)-B(z)=0\,,\\
&(z-1) z \hb(z)-\hB(z)=0\,,\\
&n\big(\tfrac12-z\big)+(n-p) (z-1) z \ha(z)-p \hA(z)=0\,,\\
&n \big(\tfrac12- z\big)+p (z-1) z a(z)+(p-n) A(z)=0\,,\\
&\frac12 \bigg(\frac1{z-1}+\frac1z\bigg)+(p-n) A(z) b(z)+(p-n) B(z) \ha(z)+(z-1) z \
b'(z)=0\,,\\
&\frac12 \bigg(\frac1{z-1}+\frac1z\bigg)-p \hA(z) \hb(z)-p a(z) \hB(z)+ (z-1) z \hb'(z)=0\,,\\
&\frac1{2 (z-1) z}+1-n+p+\frac\ell p+(p-n) a(z) A(z)+(p-n) B(z) \hb(z)+(z-1) z \
a'(z)=0\,,\\
&\frac1{2 z(z-1)} +1- p +\frac\ell{n-p}-p \ha(z) \hA(z) - p b(z) \hB(z)+ (z-1) z \ha'(z)=0\,.
\end{align*}
Thus we have eight equations with eight unknowns, but solving the above system is in general a formidable task.

The crucial observation is that this system of ODE simplifies considerably when $\ell=0$. In fact, in this case we can set $a=\ha$ and $b=\hb$, which immediately leads to the solution
\begin{align*}
&a(z)=\ha(z)=\frac{1 - 2 z}{2 z (1 - z)}\,,\qquad b(z)=\hb(z)=-\frac1{2 z (1 - z)}\,.
\end{align*}
Thus one derives that when $\ell=0$ the equations satisfied by the new functions $g$ and $\hg$ are
\begin{subequations}\label{2ODEsCCS}
\begin{align}
z(1-z)g'(z)+(n-p)\big(\tfrac12-z\big)g(z)+\tfrac 12p\hg(z)&=h(z)\,,\label{2ODEsCCSb}\\
z(1-z)\hg'(z)+p\big(\tfrac12-z\big)\hg(z)+\tfrac12(n-p)g(z)&=h(z)\,,
\end{align}
\end{subequations}
and by isolating $\hg$ in~\eqref{2ODEsCCSb} one readily finds that
\[
z(1-z)g''(z)+\big(\tfrac n2+1\big)(1-2z)g'(z)-(n-p)(p+1)g(z)=H(z)\,,
\]
with $H(z):=h'(z)+\frac{p h(z)}{z-1}$. This is a hypergeometric equation, which can be solved in terms of associated Legendre functions by the variation of constants method.

Now it suffices to notice that  the equations controlling the radial behavior of the Green's operators on $M$ are obtained from those of $M\times\Si$ by collapsing $\Si$ to a point, which amounts to setting $\ell=0$. Thus from this discussion and the previous sections we recover  the (non-rigorous) result of Allen--Jacobson~\cite{AJ86} and Folacci~\cite{Fo92} that the Green's function for the Hodge Laplacian in a simply connected space of constant curvature can be expressed in terms of hypergeometric functions, as in the scalar case. Detailed albeit somewhat formal discussions can be found in the aforementioned references; details are omitted. It should be stressed that for arbitrary values of $\ell$ the radial equation~\eqref{ODEszaCCS} do not seem to admit an analogous reduction.

\section{The radial equation on surfaces }
\label{S:dim2}

In this section we shall prove Proposition~\ref{P:spectral} and derive
explicit formulas for the elements appearing in the statement of the
result. We shall always assume that $\vka\neq0$: indeed, when
$\vka=0$ we have $L_0=L_1=L_2$ and the spectral decomposition of
this operator claimed in the latter lemma reduces to the Hankel
transform. Since $L_0=L_2$, it is obviously sufficient to consider
the cases $q=0$ and $q=1$.

We shall define a new variable $t:=\sin^2\frac{\ska s}2$, which ranges over $(0,1)$ if $\vka>0$ and over $(-\infty,0)$ if $\vka<0$. Thus we can identify the Hilbert space
$L^2((0,\diam(M)),\dd\mu_\Si)$ with $L^2((0,1),\frac{4\pi}{|\vka|}\dd t)$ (if $\vka>0$) or $L^2((-\infty,0),\frac{4\pi}{|\vka|}\dd t)$ (if $\vka<0$) via a unitary transformation $V_\vka$, and write $L_q=\vka\,V_\vka^{-1} T_q V_\vka$ with
\begin{align}\label{Tq}
T_0:=-t(1-t)\frac{\pd^2}{\pd t^2}-(1-2t)\,\frac\pd{\pd t}\,,\qquad
T_1:=T_0+\frac1{1-t}\,.
\end{align}

It is not difficult to see that the singular differential operator $T_0$ (resp.\ $T_1$) is in the limit circle case at $0$ and $1$ (resp.\ at $0$), and in the limit point case at infinity (resp.\ at $1$ and at
infinity). For simplicity of notation we shall still denote by $T_q$ the self-adjoint operators defined by the action of the differential operators~\eqref{Tq} on the domains~\cite{DS88}
\begin{subequations}\label{Dom}
\begin{align}
\Dom(T_0)&:=\Big\{u\in H^1((0,1)): T_0u\in L^2((0,1)),\; \lim_{t\downarrow0}t\,u'(t)=\lim_{t\uparrow1}(1-t)\,u'(t)=0\Big\}\,,\\
\Dom(T_1)&:=\Big\{u\in H^1((0,1)): T_1u\in L^2((0,1)),\; \lim_{t\downarrow0}t\,u'(t)=0\Big\}
\end{align}
when $\vka>0$ and
\begin{align}
\Dom(T_0)&:=\Big\{u\in H^1((-\infty,0)): T_0u\in L^2((-\infty,0)),\; \lim_{t\uparrow0}t\,u'(t)=0\Big\}\,,\\
\Dom(T_1)&:=\Big\{u\in H^1((-\infty,0)): T_1u\in L^2((-\infty,0)),\; \lim_{t\uparrow0}t\,u'(t)=0\Big\}
\end{align}
\end{subequations}
when $\vka<0$. The domains of the operators $L_q$ defined in Section~\ref{S:Riem} are simply $\Dom(L_q)=V_\vka^{-1}(\Dom(T_q))$.

When $\Si$ is compact, the spectrum of $L_q$ is discrete and its spectral decomposition can be written as follows.

\begin{lemma}\label{L:spec+}
Let us suppose that $\ka>0$. Then Proposition~\ref{P:spectral} holds with
\begin{align*}
&\dd\rho_0(\la)=\bigg(\sum_{j\geq0}\de\big(\la-\vka j(j+1)\big)\bigg)\,\dd\la\,, \qquad \dd\rho_1(\la)=\bigg(\sum_{j\geq0}\de\big(\la-\vka(j^2+3j+2)\big)\bigg)\,\dd\la\,,\\
&w_0\big(s,\vka j(j+1)\big)=\bigg(\frac{\vka(2j+1)}{4\pi}\bigg)^{1/2} P_j(\cos\ska s)\,,\\
&w_1\big(s,\vka(j^2+3j+2)\big)=\bigg(\frac{\vka(2j+3)}{4\pi}\bigg)^{1/2} \cos^2\frac{\ska s}2\,P^{(0,2)}_j(\cos\ska s)\,,
\end{align*}
where $P_\nu$ and $P^{(a,b)}_\nu$ respectively denote the Legendre and Jacobi polynomials of degree $\nu$.
\end{lemma}
\begin{proof}
It is standard that the eigenvalues of $T_q$ are nonnegative and that the normalized eigenfunctions of $T_q$ provide an orthonormal basis of $L^2((0,1))$. The choice of boundary conditions~\eqref{Dom} ensures that the eigenvalues of $L_0$ are the numbers $\La\geq0$ for which the  classical hypergeometric equation
\[
0=(\La-T_0)u=t(1-t)u''+(1-2t)u'+\La u
\]
has a polynomial solution. The solution to the latter equation which is continuous at $0$ is proportional to the hypergeometric function $F(\frac12-({\frac14+\La})^{1/2},\frac12+({\frac14+\La})^{1/2},1;t)$ and this function becomes a polynomial if and only if $({\frac14+\La})^{1/2}-\frac12\in\NN$. Hence $\spec(T_0)=\{j(j+1):j\in\NN\}$. In this case the latter hypergeometric function reduces to a Legendre polynomial, and using the well known formulas for the norm of these polynomials one readily arrives at the expression
\[
\sqrt{2j+1} \: P_j(1-2t)
\]
for the normalized eigenfunctions of $T_0$. By making use of the identification of $L^2((0,\diam(\Si)) , \dd\mu_\Si)$ with $L^2((0,1),\frac{4\pi}{\vka}\dd t)$ via $V_\vka$ we immediately arrive at the above expression for the spectral resolution of $L_0$.

The proof for $L_1$ is analogous. If we set $u(t)=:(1-t)v(t)$, the eigenvalue equation reads as
\begin{equation}\label{Jacobi}
0=(\La-T_1)u=(1-t)\big(t(1-t)v''+ (1-4t)v'+(\La-2)v\big)\,,\qquad u\in\Dom(T_1)\,.
\end{equation}
The eigenvalues of $T_1$ can be easily seen to coincide with the polynomial solutions to the latter equation. As the solution $v$ of~\eqref{Jacobi} regular at $0$ is proportional to $F(\frac32-({\frac14+\La})^{1/2},\frac32+({\frac14+\La})^{1/2},1;t)$, this implies that $({\frac14+\La})^{1/2}-\frac32\in\NN$ and $\spec(T_1)=\{j^2+3j+2:j\in\NN\}$. Hence we immediately obtain the desired formula by writing the resulting eigenfunction function in terms of a Jacobi polynomial, expressing the result in the variable $s$ and  normalizing it to have unit $L^2$ norm.
\end{proof}

When $\Si$ is noncompact, the spectrum of $L_q$ is absolutely continuous and its spectral resolution can be derived using the Weyl--Kodaira theorem.

\begin{lemma}\label{L:spec-}
Let us suppose that $\vka<0$ and set
\[
\al(\La):=\tfrac14+\sqrt{\tfrac14+\La}\,.
\]
Then Proposition~\ref{P:spectral} holds with
\begin{align*}
&\dd\rho_0(\la)=\dd\rho_1(\La)=\frac\vka{4\pi}\tanh\bigg[\pi\Big(-\frac\la\vka-\frac14\bigg)^{1/2}\bigg]\,\dd\la\,,\\
&w_0(s,\la)=\bigg(\cosh\frac{\sqrt{-\vka} s}2\bigg)^{-2\al(\la/\vka)} F\Big(\al(\la/\vka),\al(\la/\vka),1;-\tanh^2\frac{\sqrt{-\vka}s} 2\Big)\,,\\
&w_1(s,\la)=\bigg(\cosh\frac{\sqrt{-\vka} s}2\bigg)^{-2\al(\la/\vka)} F\Big(\al(\la/\vka)+1,\al(\la/\vka)-1,1;-\tanh^2\frac{\sqrt{-\vka}s} 2\Big)\,.
\end{align*}
\end{lemma}
\begin{proof}
It is a standard spectral-theoretic result that the spectrum of $T_0$, which is is absolutely continuous, is given by $(-\infty,-\frac14]$. Hence let us take a complex number $\La$ with $\Re\La\leq-\frac14$ and nonzero imaginary part and consider the equation
\begin{equation}\label{L-T0}
0=(\La-T_0)u=t(1-t)u''+(1-2t)u'+\La u\,,
\end{equation}
where $t\in\RR^-$ and $T_0$ is to be interpreted as a formal differential operator.
Then the solutions of~\eqref{L-T0} satisfying the boundary condition required in $\Dom(T_0)$ at $0$ are proportional to
\[
w(t,\La):=(1-t)^{-\al(\La)}F\Big(\al(\La),\al(\La),1;\frac t{1-t}\Big)\,,
\]
and those which are square integrable at infinity are proportional to
\begin{align*}
w^-(t,\La)&:=(1-t)^{-\al(\La)} F\Big(\al(\La),\al(\La),2\al(\La);\frac1{1-t}\Big)\qquad &\text{if }\,\Im\La<0\,,\\
w^+(t,\La)&:=(1-t)^{\al(\La)-1} F\Big(1-\al(\La),1-\al(\La),2-2\al(\La);\frac1{1-t}\Big)\qquad &\text{if }\,\Im\La>0\,.
\end{align*}
For notational simplicity we shall henceforth write $\al$ instead of
$\al(\La)$. It should be noticed that the determination of the
square root ensures that these functions depend analytically on
$\La$ in the region $\Re\La<-\frac14$. It is well known~\cite{AS70}
that one can write e.g.\ $w^+$ as a linear combination of $w$ and
$w^-$ as $w^+(t,\La)=k_1(\La)\, w(t,\La)+ k_2(\La)\, w^-(t,\La)$,
with
\[
k_1(\La):=\frac{\Ga(\al)^2}{\Ga(2\al-1)}\,,\qquad k_2(\La):=-\frac{\Ga(\al)^2\Ga(1-2\al)}{\Ga(1-\al)^2\Ga(2\al-1)}\,.
\]

Let us use the notation
\[
\widetilde W(u,v):=t(1-t)\,\big( u'(t)v(t)-u(t)v'(t)\big)
\]
for the reduced Wronskian of two solutions $u$ and $v$ of Eq.~\eqref{L-T0}, which is actually constant. The Weyl--Kodaira theorem~\cite{DS88} can be used to prove that the self-adjoint operator $T_0$ admits a spectral decomposition analogous that of Proposition~\ref{P:spectral}, where $w$ plays the role of the function $w_q$ and $\dd\rho_q$ must be replaced by the absolutely continuous Borel measure
\[
\dd\rho(\La):=\frac{k_1(\La)}{2\pi\I\,\widetilde W(w(\cdot,\La),w^+(\cdot,\La))}\dd\La
\]
on $(-\infty,-\frac14]$. By~\cite[15.3.10]{AS70} one has that the function $w^+(\cdot,\La)$ can be written as
\[
w^+(t,\La)=-\frac{\Gamma(2-2\al)}{\Gamma(1-\al)^2}\log(-t)+\varphi(t,\La)\,,
\]
where $\varphi(\cdot,\La)$ is of class $C^1$ in a neighborhood of $0$. As $\widetilde W(w(\cdot,\La),w^+(\cdot,\La))$ is constant, we immediately arrive at the formula
\[
W(w(\cdot,\La),w^+(\cdot,\La))=\frac{\Gamma(2-2\al)}{\Gamma(1-\al)^2}
\]
for the reduced Wronskian, which in turn yields the expression
\begin{equation}\label{rho}
\dd\rho(\La)=\tanh\bigg[\pi\sqrt{-\La-\tfrac14}\bigg]\,\dd\La
\end{equation}
for the measure $\rho$. If now use the identification of $L^2((0,\infty),\dd\mu_\Si)$ with $L^2((-\infty,0),\frac{4\pi}\vka\dd t)$ we readily arrive at the formulas for the spectral decomposition of $L_0$.

The analysis of $L_1$ is similar. Setting $u(t)=:(1-t)v(t)$ in the equation $(\La-T_1)u=0$ as in Lemma~\ref{L:spec+} we find that the regular solution of this equation at $0$ is
\[
w(t,\La):=(1-t)^{-\al}F\Big(\al+1,\al-1,1,\frac t{1-t}\Big)
\]
up to a multiplicative constant, and that the solution which is square integrable at infinity is proportional to
\begin{align*}
w^-(t,\La)&:=(1-t)^{-\al} F\Big(\al+1,\al-1,2\al;\frac1{1-t}\Big)\qquad &\text{if }\,\Im\La<0\,,\\
w^+(t,\La)&:=(1-t)^{\al-1} F\Big(2-\al,-\al,2-2\al;\frac1{1-t}\Big)\qquad &\text{if }\,\Im\La>0\,.
\end{align*}
Since the reduced Wronskian of $w$ and $w^+$ is
\[
\widetilde W(w(\cdot,\La),w^+(\cdot,\La))=\frac{\Ga(2-2\al)}{\Ga(2-\al)\Ga(-\al)}
\]
and
\[
w^+(t,\La)=\frac{\Ga(\al+1)\Ga(\al-1)}{\Ga(2\al-1)}\bigg[w(t,\La)-\frac{\Ga(1-2\al)}{\Ga(-\al)\Ga(2-\al)}w^-(t,\La)\bigg]\,,
\]
the same reasoning as above shows that the spectral measure $\dd\rho$ associated with $T_1$ is also given by Eq.~\eqref{rho}.
\end{proof}

As stated in Proposition~\ref{P:spectral}, the integral
representation~(\ref{L2conv}) converges to $g(L_q)u$ in the norm
topology, and by Egorov's theorem this implies that the latter
integral converges uniformly except on a subset of arbitrarily small
measure. We find it convenient to conclude this section with another
simple but useful observation concerning the pointwise convergence
of previously defined the spectral decompositions.

\begin{lemma}\label{L:unif}
Let $U_q$, $w_q$, $c_{q,\vka}$ and $\rho_q$ be defined as in Proposition~\ref{P:spectral}, and let $u\in C^\infty_0([0,\diam(\Si)))$. Then
\begin{equation}\label{L2conv2}
u(s)=\int_{c_{q,\vka}}^{\infty} w_q(s,\la)\,U_qu(\la)\,\dd\rho_q(\la)
\end{equation}
pointwise for all $s\in[0,\diam(\Si))$, and the convergence is uniform.
\end{lemma}
\begin{proof}
For the sake of concreteness we restrict ourselves to the case $\vka>0$. As $C^\infty_0([0,1)) \subset\Dom(L_q^m)=V_\vka^{-1}(\Dom(T_q^m))$ for any nonnegative integer $m$, it follows that the integral
\begin{equation}\label{boundUu}
(L_q^mu,u)=\int_{c_{q,\vka}}^{\infty} \la^m \, \big|U_qu(\la)\big|^2\,\dd\rho_q(\la)<\infty
\end{equation}
must convergent for all $m$ and $u\in C^\infty_0([0,\diam(\Si)))$.

It is well known~\cite{AS70} that the Legendre and Jacobi polynomials satisfy
\begin{equation*}
\big|P_j(\xi)\big|\leq1\,,\qquad \big|P_j^{(0,2)}(\xi)\big|\leq\frac{(j+2)(j+1)}2
\end{equation*}
for all $\xi\in[-1,1]$. By the explicit expression for $w_q$, this immediately implies that
\begin{align*}
\bigg|\int_{c_{q,\vka}}^{\infty} w_q(s,\la)\,U_qu(\la)\,\dd\rho_q(\la)\bigg|&\leq \int_{c_{q,\vka}}^{\infty}(c_1\la+c_2)^{1/2}\big|U_qu(\la)\big|\,\dd\rho_q(\la)<\infty\,,
\end{align*}
which converges by virtue of Eq.~\eqref{boundUu}. Here $c_1$ and $c_2$ are positive constants. Hence the integral~\eqref{L2conv2} converges uniformly if the support of $u$ is contained in $[0,K]$ with some $K<\diam(\Si)$. (The same argument also yields convergence in the $C^k$ strong topology). When $\vka\leq0$ the proof is analogous and will be omitted.
\end{proof}

\section*{Acknowledgements}

A.E.\ acknowledges the partial financial support by the DGI and the
Complutense University--CAM under grants no.~FIS2005-00752
and~GR69/06-910556, and thanks McGill University, Montr\'eal, for
hospitality and support. The research of N.K. is supported by NSERC
grant RGPIN 105490-2004.

\end{document}